\documentclass[a4paper,10pt, smallheadings]
{scrartcl}
\usepackage{amsmath}
\usepackage{amsfonts}
\usepackage{amssymb}
\usepackage{amsthm}
 \usepackage{verbatim}
\usepackage{fancyhdr}
\usepackage{dsfont}
\usepackage[dvips]{graphicx}
\usepackage{setspace}
\usepackage[ansinew]{inputenc}
\newcommand\tr{\mathop{\rm tr}\nolimits}
\newcommand\diff{\mathop{\rm Diff}\nolimits}
\newcommand\Lip{\mathop{\rm Lip}\nolimits}

\newcommand{\R}{\mathbb{R}}
\newcommand{\C}{\mathbb{C}}

\newcommand{\Y}{\mathcal{Y}}

\renewcommand{\L}{\mathcal{L}}
\newcommand{\A}{\textbf{A}}
\newcommand{\B}{\textbf{B}}
\renewcommand{\epsilon}{\varepsilon}
 
\newtheorem{thm}{Theorem}[section]
 \newtheorem{cor}[thm]{Corollary}
 \newtheorem{lem}[thm]{Lemma}
\newtheorem{prop}[thm]{Proposition}
 \theoremstyle{definition}
 
 \theoremstyle{remark}

\begin{document}
%
%
%
%
%
%
%
%
%
\title
 {Strong solutions of semilinear matched micro\-structure models}

\author{Joachim Escher and Daniela Treutler}



\date{}
\maketitle

\begin{abstract}
\textbf{Abstract.} The subject of this article is a matched microstructure model for Newtonian fluid flows in fractured porous media. This is a homogenized model which takes the form of two coupled parabolic differential equations with boundary conditions in a given (two-scale) domain in Euclidean space. The main objective is to establish the local well-posedness in the strong sense of the flow. 
Two main settings are investigated: semi-linear systems with linear boundary conditions and semi-linear systems with nonlinear boundary conditions. With the help of analytic semigoups we establish local well-posedness and investigate the long-time behaviour of the solutions in the first case: we establish global existence and show that solutions converge to zero at an exponential rate.\\[0.2cm]
\let\thefootnote\relax
\let\thefootnote\relax
\end{abstract}

\section{Introduction}
First ideas to use a two scale formulation to describe fluid flow in fractured porous media came up around 1960 e.g. by Barenblatt, Zheltov and Kochina \cite{BZK}. They reflect very well the exceptionell manner of the geometric conditions. The material possesses one natural porous structure and a second one is added by the dense system of cracks. A mathematical derivation of the model was later given in \cite{ADH}. In the sequel weak formulations of the problem have been studied intensively either by homogenization theory \cite{Hornung1, Hornung2} or with the help of monotone operators \cite{Antontsev,  CookShow, HorShow}. Stationary solutions and the elliptic problem are for example treated in \cite{ShoWa3}. Reaction terms and evolving pore geometry have been considered in several papers \cite{FrieKna, FrieTza, Meier1, MeiBo, Peter1, PeterBoehm}. Although the model in these studies has a similar form as in the present paper, the considered length scales are different. 

We consider the matched microstructure model (MM) as it was formulated by Showalter and Walkington in \cite{ShoWa}. Assume we are given a \textit{macroscopic} domain $\Omega\subset \R^n$, and for each $x\in \Omega$ a \textit{cell} domain $\Omega_x\subset \R^n$. These cell domains stand for the porous blocks while $\Omega$ contains in a homogenized sense the fissure system.

The model (MM) consists of two parts: The macro model for a function $u$, that represent the density of the fluid on the domain $\Omega$:
\begin{align*}
 \frac{\partial}{\partial t} u(t,x)-\Delta_x u(t,x) &= f(t,u)+ q(U)(t,x), && x\in \Omega, t\in (0,T],\\
u(t,x)&=0, && x\in \Gamma, t\in (0,T],\\
u(t=0)&=u_0.
\end{align*}
The micro model for the function $U$, that models the density in all blocks $\Omega_x$:
\begin{align*}
\frac{\partial}{\partial t} U(t,x,z)-\Delta_z U(t,x,z)&=0, &&x\in \Omega, z\in \Omega_x, t\in (0,T],\\
U(t,x,z)&=u(x), &&x\in \Omega, z\in  \Gamma_x, t\in (0,T],\\
U(t=0)&=U_0.
\end{align*}
The coupling between the macro and the micro scale is 
reflected by two terms. Firstly the boundary condition in the cells $\Omega_x$,
\begin{align}
 U(x)=u(x) \qquad \text{on } \partial \Omega_x, \text{ for all } x\in \Omega\label{match}
\end{align}
models the matching of the densities on the material interface.
For this reason the model was introduced in \cite{ShoWa} as matched microstructure model. Secondly the term
\begin{equation}\label{qU}
 q(U)(t,x) = -\int_{\Gamma_x} \frac{\partial U(t,x,s)}{\partial \nu} \, ds = -\frac{\partial}{\partial t} \int_{\Omega_x} U(t,x,z) \, dz.  
\end{equation}
represents the amount of fluid that is exchanged between the two structures. It acts as a source or sink term in the
 macroscopic system. 

Our interpretation of this model is based on the derivation of the coupled equations for the case of uniform cells at each
 point $x$ in the considered domain $\Omega$. On this basis we will first present our restrictions on the geometry and the definition of suitable Banach spaces. Then we reformulate the problem as an abstract semilinear intial value problem on the product space $ L_p(\Omega)\times L_p(\Omega, L_p( \Omega_x))$. Therefore we introduce an operator $\A$ that includes the highest order derivatives and the first coupling condition.
 In Theorem \ref{allgemein} we prove that $-\A$
 generates an analytic semigroup which finally implies well-posedness of the matched microstructure problem. 

A further part of this work is the consideration of the long time behaviour of the solution. We show that for Dirichlet boundary condition the solution decays to zero at an exponential rate. In the Neumann case we prove mass conservation.
Finally we consider a special two dimensional geometry and include a nonlinear boundary condition on $\partial \Omega$. A detailed derivation of this is given in \cite{Doktorarbeit}.
 We prove well posedness for weak solutions which can be improved using the concept of Banach scales. The approach is funded on work of Escher \cite{Escher89} and Amann \cite{Amann88, Amann_Multi, Amann}. Our methods are quite flexible and provide the high regularity of strong solutions.

The outline of this paper is the following. First we present some basic lemmas for a uniform geometry. The largest section is then devoted to the variation of the cell domain and the transcription of the problem (MM) into an abstract initial value problem. There the mains result that $-\A$ generates an analytic semigroup is located. Further the spectrum of $\A$ is investigated for the case $p=2$. The last chapter contains the model which includes nonlinear boundary conditions. With the help of a retraction from the boundary we transform the equations into an abstract semilinear problem, which leads to well-posedness of the original system.

\section{Some Aspects for Uniform Cells}
Let $\Omega\subset \R^n$ be a bounded domain with smooth boundary $\Gamma:= \partial \Omega$ and let $(\Omega, \mathcal{A}, \mu)$ be a measure space. 
Let $X,Y$ be Banach spaces. If $A$ is a closed linear operator from $X$ to $Y$ we denote with $D(A) = (dom(A), \|\cdot\|_A)$
the domain of definition of $A$ equipped with the graph norm. Further we write $\hat{U}$ if we mean a representativ in $\mathcal{L}_p$ of a give function $U\in L_p$. With $[\cdot]$ we indicate the equivalence class again.
The shifted operators will always be denoted with bold letters.
\begin{lem}\label{boundedShift}
Assume that $(x\mapsto A(x)) \in C(\overline{\Omega}, \mathcal{L}(X,Y))$. Let
\begin{align*}
 dom(\A) &= L_p(\Omega, X),\\
\A U &= \left[A(x)\hat{U}(x)\right], && \text{for } U\in L_p(\Omega, X),\, \hat{U} \in U.
\end{align*}
Then $\A$ is a well defined, bounded linear operator from $L_p(\Omega, X)$ to $L_p(\Omega, Y)$.
If further $A(x)=A$ independent of $x$ and $A$ is a retraction, then $\A$ is a retraction as well. 
\end{lem}
\begin{proof}
The continuity of $(x\mapsto A(x))$ assures that $\A$ is well defined.  
Further we easily get $\A U\in L_p(\Omega, Y)$ and $\A$ is bounded because
\begin{align*}
 \|\A U\|^p_{L_p(\Omega,X)} =\int_\Omega \|A(x) \hat{U}(x)\|^p_Y \, d\mu(x) \leq \max_{x\in \overline{\Omega}} \|A(x)\|^p_{\L(X,Y)} \cdot\|U\|^p_{L_p(\Omega, X)}.
\end{align*}
Now assume that  $A(x)=A$ is a retraction and let $R \in \L(Y,X)$ be a continuous right inverse of $A$, so that $ A \circ R = id_Y.$
Let $V\in L_p(\Omega,Y)$, $\hat{V}\in \L_p(\Omega,Y)$ a representative of $V$. As before we define 
$$\textbf{R}\in \L(L_p(\Omega, Y),L_p(\Omega, X)),\qquad \textbf{R}V = [R\hat{V}(x)].$$
Then a short calculation shows that this is a continuous right inverse for $\A$.
\end{proof}
\begin{lem}\label{Lemma2}
For $x\in \Omega$, let $A(x) \in \mathcal{A}(X,Y)$ be a closed linear operator. Assume that there is $A_0\in \mathcal{A}(X,Y)$, such that $dom(A(x)) = dom(A_0),\text{ for all }x\in \Omega.$
Furthermore let $(x\mapsto A(x)) \in C(\overline{\Omega}, \L(D(A_0),Y))$. Then the operator
\begin{align*}
 dom(\A) &= L_p(\Omega, dom(A_0)),\\
\A U &= \left[A(x)\hat{U}(x)\right], && \text{for } U\in dom(\A),\, \hat{U} \in U.
\end{align*}
is a well defined, closed linear operator from $L_p(\Omega, X)$ to $L_p(\Omega, Y)$. If further $dom(A_0)$ is dense in $X$,
 then $\A$ is densely defined.
\end{lem}
\begin{proof}
The first statement follows from Lemma \ref{boundedShift} and a short calculation. The density assertion follows from the proof of Theorem 4.3.6 in \cite{Amann}.
\end{proof}
Now we turn to sectorial operators.  
Let $\omega \in \R, \theta \in (0, \pi)$. We set
$$S_{\theta,\omega} = \{\lambda \in \C; \lambda \neq \omega, |arg(\lambda-\omega)| < \theta\}.$$ 
\begin{lem}\label{Lemma3}
 Let the assumptions of the previous Lemma be fulfilled with $X=Y$. Assume further that there exist constants
 $\omega \in \R, \theta \in (0, \pi)$, $M\geq 1$ such that for every $x\in \Omega$, $S_{\theta,\omega} \subset \rho(-A(x))$ and 
\begin{align*}
 \|(\lambda +A(x))^{-1}\|_{\L(X)} &\leq \frac{M}{|\lambda-\omega|} && \text{for } \lambda \in S_{\theta,\omega}.  
\end{align*}
Then $\A$ is sectorial in $L_p(\Omega, X)$.
\end{lem}
\begin{proof}
 Let $\lambda \in S_{\theta,\omega}$. With Lemma \ref{boundedShift} we define $\textbf{R}_\lambda\in \L(L_p(\Omega, X))$ by
$$\textbf{R}_\lambda U := [(\lambda + A(x))^{-1} \hat{U}(x)]$$
for $U\in L_p(\Omega, X)$. Then one easily calculates that this is the inverse of $\lambda+ \A$. Thus $\lambda \in \rho(-\A)$. Furthermore it holds
\begin{align*}
 \|(\lambda+\A)^{-1}U\|_{L_p(\Omega,X)}^p&\leq \int_\Omega \|(\lambda + A(x))^{-1} \|^p_{\L(X)} \|U(x)\|_X^p \, d\mu(x)\\
&\leq \left(\frac{M}{|\lambda-\omega|}\right)^p \int_\Omega \|U(x)\|_X^p \, d\mu(x). 
\end{align*}
So we conclude that $\A$ is sectorial.
\end{proof}
\section{The Semilinear Problem}
\subsection{Geometry}\label{1.1}
 The main idea of this part is to relate the cell's shape to one standard cell, the unit ball $B= B(0,1) \subset \R^n$. Let $S= \partial B$ be its boundary. 
We assume that there are two mappings $\Psi, \Phi$ with
\begin{align*}      
 \Psi : \Omega \times B &\to \mathbb{R}^n,\\
 \Phi: \Omega\times B &\to \mathbb{R}^n\times\mathbb{R}^n,\\
  (x,y) &\mapsto (x, \Psi(x,y)).
\end{align*}
Now a cell at a point $x\in \Omega$ is the image of $B$ at $x$, i.e. $\Omega_x:= \Psi(x,B).$ We set
$$Q:= \bigcup_{x\in \Omega}  \{x\} \times \Omega_x.$$
Then $ Q= \Phi(\Omega\times B)$.
To assure that $\Omega_x$ is a bounded smooth domain as well impose some properties of $\Phi, \Psi$:
\begin{align}
 \Phi&\in \Lip(\Omega\times B, Q),\label{cond1}\\ 
 \Phi^{-1} &\in \Lip( Q, \Omega\times B), \label{cond2}\\
 \Phi(x, \cdot) &\in \diff(\overline{B}, \overline{\Omega}_x), && \textrm{for all }x\in \Omega,\label{cond3}\\ 
 &\sup_{x\in \Omega, |\alpha|\leq 2}   \left\{\|\partial^\alpha_y \Phi(x)\|_p, \|\partial^\alpha_z \Phi^{-1}(x)\|_p\right\} < \infty.\label{cond4}
\end{align}
Here $\|\cdot\|_p$ denotes the usual $L_p$-norm. The set $\diff(\overline{B}, \overline{\Omega}_x)$ shall denote all $C^\infty$-diffeomorphism from $\overline{B}$ to $\overline{\Omega}_x$ such that the restriction to the boundary $S$ gives a diffeomorphism to $\Gamma_x$. It follows from the assumptions that $Q$ is measurable. Further for every $x\in \Omega$, the set $\Omega_x$ is a bounded domain with smooth boundary $\Gamma_x:= \partial \Omega_x$. Note that the special construction of $\Phi$ 
implies that it is injectiv.  Thus we will be able to work with the trace operator on $B$ and transfer it to $\Omega_x$. The conditions ensure that the following maps are well defined isomorphisms. Given $2\leq p<\infty$, we define pull back 
and push forward operators
\begin{align*}
 \Phi_* &: L_p(\Omega\times B)\to L_p(Q):
U  \mapsto U \circ \Phi^{-1},\\
\Phi^* &: L_p(Q)\to L_p(\Omega \times B):
 V \mapsto V\circ \Phi. 
\end{align*}
The following definition of a function space is based on Bochner's integration theory. 
In \cite{Adams} (3.34, 3.35), it is proven for Sobolev-Slobodetski spaces that under these diffeomorphisms $W_p^s(\overline{B})$ is mapped onto $W_p^s(\overline{\Omega}_x)$ for $0\leq s\leq 2$. For $s=0$ we identify $W_p^0(\overline{B}) =L_p(B)$.
The space  $L_p(\Omega, W_p^s(B))$ is now defined by means of the Bochner integration theory.
 We define
\begin{align}
 L_p(\Omega, W_p^s(\Omega_x)) := \Phi_* (L_p(\Omega, W^s_p(B))).\label{spaces}
\end{align}
We can prove that equipped with the induced norm 
\begin{align*}
 \|f\|_{x,s}&:= \|\Phi^* f\|_{L_p(\Omega, W_p^s(B))}, &&f\in L_p(\Omega, W_p^s(\Omega_x)),
\end{align*}
this is a Banach space.\footnote{Note that hypothesis \eqref{cond4} ensures that different $\Phi$'s within the class \eqref{cond1} to \eqref{cond4} lead to equivalent norms.}
For the formulation of the boundary conditions on the cells we need suitable trace operators. 
Our assumptions ensure that we can restrict $\Phi^*, \Phi_*$ to $L_p(\Omega\times S)$.
We use the same notation for the pullback and push forward as on $\Omega\times B$. Let $s\geq 0$. We define
\begin{align*}
L_p(\Omega, W_p^s(\Gamma_x))&:= \Phi_* \left(L_p(\Omega, W_p^s(S))\right)\\
\|U\|_{L_p(\Omega, W^s_p(\Gamma_x))} &= \|\Phi^* U\| _{L_p(\Omega, W^s_p(S))}, && U\in L_p(\Omega, W^s_p(\Gamma_x)).
\end{align*}
As before this is a Banach space. From Lemma \ref{boundedShift} we deduce that the shifted trace
\begin{align*}
 \tr_S: L_p(\Omega, W_p^1(B)) \to L_p(\Omega, W_p^{1-\frac1p} (S)): \tr_S U = [tr_S\hat{U}],
\end{align*}
is a well defined linear operator. The last trace in the brackets is the usual trace on $B$. Next we transport this operator to $Q$. We set
\begin{align*}
 \tr &:L_p(\Omega, W_p^1(\Omega_x)) \to L_p(\Omega, W_p^{1-\frac1p}(\Gamma_x)), \,
\tr:= \Phi_* \tr_S \Phi^*.
\end{align*}
The continuity of $\Phi_*, \Phi^*, \tr_S$ ensures that $\tr$ is a continuous operator. 
In particular $\tr U=0$ implies $\tr_S ( \Phi^* U)=0$. 
From Lemma \ref{boundedShift} we conclude that $\tr_S$ is a retraction. There exists a continuous right inverse $R_S$ of $\tr_S$  that maps constant functions on the boundary to constant functions on $B$. We define
\begin{align*}
R&:= \Phi_* R_S \Phi^*. 
\end{align*}
 Then this is a continuous right inverse to $\tr$. Let $u\in L_p(\Omega)$. We identify 
$$R u = R(u \cdot 1_S)\in L_p(\Omega, W_p^2(B)).$$
With $\Delta_z$ we denote the Laplace operator in the coordinates $z\in \Omega_x$. Similarly we write $\Delta_y$ and $\Delta_x$ for the Laplace acting on functions over $B$ or $\Omega$. The definitions above ensure that 
\begin{align}
 \Delta_z Ru(x) =0, \qquad \text{for a.e. } x\in  \Omega.\label{null}
\end{align}
This will be helpful in later calculations. Another definition of function spaces for the matched microstructure problem can be found in \cite{MeiBo}.
\subsection{Operators} \label{TransOp}
To use existing results for strongly elliptic operators, we first consider some auxiliary operators. Let $A_1$ be the Dirichlet-Laplace operator on $\Omega$,
\begin{align*}
 dom(A_1)&= W_p^2(\Omega)\cap W_p^{1,0}(\Omega),\\
A_1 u&=-\Delta_x u, && \text{for } u\in dom(A_1).
\end{align*}
It is well known that $A_1$ is sectorial. 
For each $x\in \Omega$ we define a Riemannian metric $g(x)$ on the unit ball $B$. We write
\begin{align*}
 g_{ij}(x) &:=(\partial_{z_i} \Phi(x)|\partial_{z_j} \Phi(x)),\\
 \sqrt{|g(x) |}&:= \sqrt{\det{g_{ij}(x)}},\\
 g^{ij}(x) &:= (g_{ij}(x))^{-1}.
\end{align*}
Then the regularity assumptions on $\Phi$ imply that this metric is well defined and there exists constants $C_i>0, i=1,2$ such that $C_1 |x|^2 \leq \sum_{i,j} g^{ij} x_i x_j \leq C_2   |x|^2$.
Let $U\in L_p(\Omega, W_p^1(\Omega_x))$, $V= \Phi^* U$. We set
\begin{align}
 q(U)(x) := -\int_S \sqrt{|g(x)|}g^{ij}(x) \partial_{y_i} \hat{V}(x) \cdot  \nu_j \, ds.
\end{align}
Here $\nu=(\nu_1, \dots , \nu_n)$ denotes the outer normal vector on $B$. Then $q(U)$ is a function in $ L_p(\Omega)$. Using the transformation rule for integrals one sees that this definition is consistent with  \eqref{qU}. 
We define the operator $\A_2$ using the transformed setting
\begin{align}
 dom(\A_2) &= \{U \in L_p(\Omega, W_p^2(\Omega_x)); \tr U=0\},\\
 \A_2U &= \Phi_* [\mathcal{A}_x \hat{V}(x)].
\end{align}
The brackets $[\cdot]$ again indicate taking the equivalence class and $\hat{V}$ is a representative of $V$. Given $x\in \Omega$, the operator $\mathcal{A}_x$ acts in the following way on $v\in W_p^2(B)$,
$$\mathcal{A}_x v = -\frac{1}{\sqrt{|g(x)|}} \sum_{i,j}\partial_{y_i}\left(\sqrt{|g(x)|} g^{ij}(x)  \partial_{y_j}\right) v.$$
Note that $\mathcal{A}_x$ is the Laplace-Beltrami-operator with respect to the Riemannian metric $g$. It holds
\begin{lem}\label{Lemma6}
 The operator $\A_2$ is well defined.
\end{lem}
\begin{proof}
The coefficients of $\mathcal{A}_x$ depend continuously on $x$. Moreover the domain of definition is independent of $x\in \Omega$. Since $\Phi$ is defined up to the boundary of $\Omega$ the definition can be extended to its closure. So the hypothesis follows from Lemma \ref{Lemma2} and the properties of $\Phi_*$.
\end{proof}
The following lemma collects some properties of the defined operators. Let $R(\lambda, A) = (\lambda+A)^{-1}$ denote the resolvent operator of $-A$ for $\lambda\in \rho(-A)$.
\begin{lem}
 Assume that for any $x\in \Omega$, $\Phi_x:= \Phi(x,\cdot)$ is orientation preserving. Further assume that the Riemannian metric $g^{ij}$ induced from $\Phi$ is well defined. 
For each cell we define the 
 transformation $B_x$ of the Dirichlet-Laplace operator ,
\begin{align*}
dom(B_x) &= W_p^2(B)\cap W_p^{1,0}(B),\\
 B_x v &= \mathcal{A}_x v, && \textrm{ for } \ v\in dom(B_x).
\end{align*}
It holds
\renewcommand{\labelenumi}{(\alph{enumi})}
\begin{enumerate}
 \item The operators $B_x$ are strongly elliptic in $L_p(B)$.
\item Given $x\in \Omega$, the operator $B_x$ is sectorial. In addition there exists a sector
$$S_{\theta, \omega} = \{\lambda\in \C; \lambda\neq \omega, |arg (\lambda-\omega)| < \theta\},$$
and a constant $M_2>0$, both independent of $x$, such that
\begin{align}
 \rho(-B_x) & \supseteq S_{\theta, \omega},\\
 \| R(\lambda, B_x)\|_{\L(L_p(B))} &\leq \frac{M_2}{|\lambda-\omega|}, && \text{for all } \lambda \in S_{\theta,\omega}.\label{sect}
\end{align}
\item The operator $\B$ in $L_p(\Omega \times B)$, given by
\begin{align*}
 dom(\B) &= L_p(\Omega, W_p^2(B)\cap W_p^{1,0}(B) ),\\
 \B V &= [ B_x \hat{V}(x) ], && V\in dom(\B ),\, \hat{V} \in V,
\end{align*}
is well defined and sectorial.
 \item Set $\tilde{f} := \Phi^* f\in L_p(\Omega\times B)$. If the function $V\in L_p(\Omega, W_p^2(B)\cap W_p^{1,0}(B))$ 
is a solution of $\B V=\tilde{f}$, then $U:= \Phi_* V$ fulfills
$$ -\Delta_z V(x,\cdot) = f(x, \cdot), \ \ \text{ for a.e. }x\in \Omega.$$
Moreover
$$U(x,z) = 0, \ \ \text{ for a.e. }x\in \Omega, z\in \Gamma_x.$$
\end{enumerate}
\end{lem}
\begin{proof}
The first part follows from the fact that strong ellipticity is preserved under transformation of coordinates. Part b) is a consequence of the definition of $\Phi$ and \cite{Luna}, Theorem 3.1.3. With the help of Theorem 9.14, \cite{GilbTru} we conclude
that the sector is independent of $x$. Now the rest follows by definition and Lemma \ref{Lemma6}.  
\end{proof}
A more detailed proof can be found in \cite{Doktorarbeit}.
Now we are ready to treat the coupled problem. Given $u\in L_p(\Omega)$, we set
$$D_0(u):= \left\{ U\in L_p\left(\Omega, W_p^2(\Omega_x)\right); \tr U=u\right\}.$$
This is a closed linear subspace of $L_p(\Omega, W_p^2(\Omega_x))$. So we can define the
 operator $\textbf{A}$ by
\begin{align*}
 dom(\textbf{A})&=\bigcup_{u\in 
W_p^2(\Omega)\cap W_p^{1,0}(\Omega)} \{u\} \times D_0(u),\\
\textbf{A} (u,U)&=\left(-\Delta_x u, [\Phi_* \mathcal{A}_x \Phi^* \hat{U}(x)] \right),&& \text{for } (u,U)\in dom(\textbf{A}).
\end{align*}
Observe that the operator contains the matching condition \eqref{match}. The exchange term $q(U)$ will appear as a term on the right hand side of the abstract problem \eqref{P}.
Let $(f,g) \in L_p(\Omega)\times L_p(\Omega,L_p(\Omega_x))$, $\lambda \in S_{\theta,\omega}$. We consider the system
\begin{align}
 (\lambda + \A) (u,U) = (f,g), \qquad \text{ for } (u,U)\in dom(\A).\label{problem}
\end{align}
This formally corresponds to
\begin{align}
 \lambda u-\Delta_x u&=f,&& u\in W_p^2(\Omega)\cap W_p^{1,0}(\Omega),\label{7}\\
\lambda U-\Delta_z U &=g,&& U\in D_0(u).\label{8}
\end{align}
\begin{prop}\label{Prop8}
 The operator $-\A$ is the generator of an analytic semigroup on the space $L_p(\Omega)\times L_p(\Omega,L_p(\Omega_x))$.
\end{prop}
\begin{proof}
Let $\omega_i, \theta_i$ such that $S_{\theta_1,\omega_1} \subset \rho(-A_1)$, $S_{\theta_2,\omega_2} \subset \rho(-\A_2)$. Set
\begin{align*}
 \omega &= \max\{\omega_1, \omega_2\},&& \theta = \min \{\theta_1,\theta_2\}.
\end{align*}
Then $S_{\theta, \omega} \subset  \rho(-A_1)\cap \rho(-\A_2)$. Take $\lambda \in S_{\theta,\omega}$. Without restriction we suppose $\omega=0$. Since $A_1$ is sectorial the function $u=R(\lambda, A_1)f$ solves \eqref{7}. Furthermore there is $M_1\geq 1$ such that 
$$|\lambda| \|u\| \leq \||\lambda| R(\lambda, A_1) f\|\leq M_1 \|f\|.$$
For $U \in D_0(u)$, it holds
$$ U-Ru\in L_p(\Omega,W_p^2(\Omega_x))\cap \ker \tr=dom(\A_2).$$
Here $R$ is the extension operator defined in the previous chapter. So \eqref{null} implies that \eqref{8} is equivalent to
\begin{align}
 \lambda(U-Ru)+ \A_2(U-Ru)=g-\lambda Ru.\label{9}
\end{align}
Since $\B$ is sectorial, \eqref{9} has the unique solution 
$$U=\Phi_* R(\lambda, \B)\Phi^* (g-\lambda Ru)+Ru.$$
So we have shown that (\ref{problem}) has a unique solution for $\lambda\in S_{\theta,\omega}$. 
Hence we conclude that $\lambda \in \rho(-\textbf{A})$. To shorten the notation, we write $X_0 = L_p(\Omega,L_p(\Omega_x))$. It holds
\begin{align*}
 |\lambda| \| U\|
&\leq |\lambda| \|\Phi_* R(\lambda,\B)\Phi^* g\|_{X_0} + |\lambda | \|\Phi_* R(\lambda,\B)\Phi^*\lambda R\ R(\lambda, A_1)f\|_{X_0} \\ &\qquad +|\lambda|  \|R\ R(\lambda, A_1)f\|_{X_0}\\
&\leq M_2\|\Phi_*\| \|\Phi^*\| \|g\|_{X_0} +M_2\|\Phi_*\| \|\Phi^*\|\|R\| M_1\|f\|_{L_p(\Omega)} + M_1 \|R\| \|f\|_{L_p( \Omega)}.
\end{align*}
With this we estimate the norm of the resolvent of $-\textbf{A}$
\begin{align*}
 |\lambda| \|R(\lambda,\textbf{A})\|&=\sup\{ |\lambda| \|U\| + |\lambda|\|u\| ;u=R(\lambda, A_1)f,\\
& \qquad \qquad U=R(\lambda, \A_2)(g-\lambda Ru)+Ru, \|f\|+\|g\|\leq 1\},\\
&\leq M_2\|\Phi_*\| \|\Phi^*\| + M_1 M_2\|\Phi_*\| \|\Phi^*\| \|R\| + M_1\|R\| + M_1=: M.
\end{align*}
Hence the sector is contained in the resolvent set $\rho(-\textbf{A})$ and the inequality above holds for some constant $M\geq 1 $ independent of $\lambda$. Hence $\textbf{A}$ is sectorial and $-\A$ generates a holomorphic semigroup.
\end{proof}
We are now prepared to write the matched microstructure problem as an abstract evolution equation. Set $w=(u,U)$, $w_0=(u_0,U_0)$.
We look for $w$ satisfying
\begin{align}\label{P}
 \begin{cases}
  \partial_t w + \A w  =f(w),& t\in(0,T)\\
 w(0)= w_0.&
 \end{cases}\end{align}
To solve this semilinear problem we take $\frac12<\Theta<1$. Our goal then is to show that
$$f: (0,T) \times [Y_0, D(\A)]_\Theta \to Y_0:=L_p(\Omega)\times L_p(\Omega,L_p(\Omega_x)) $$
is locally H\"older continuous in $t$ and locally Lipschitz continuous in $w$. For the initial value we will require that $w_0\in [Y_0,D(\A)]_\Theta$. Then results from Amann \cite{NLQ} imply local existence and uniqueness.
\subsection{Interpolation and Existence Results}\label{3.3}

Let $X_0, X_1$ be two Banach spaces that form an interpolation couple.
Let $0\leq\Theta\leq  1$. We denote with $[X_0,X_1]_\Theta$ the
 complex interpolation space of order $\Theta$. 
Let $\tr_{\Gamma}: W_p^1(\Omega) \to W_p^{1-\frac1p}(\Gamma)$ be the trace operator on $\Omega$. 
R. Seeley showed in \cite{Seeley} that 
\begin{align}\label{15}
 \left[L_p(\Omega), W_p^2(\Omega) \cap \ker \tr_{\partial \Omega} \right]_\Theta = \begin{cases} W_p^{2\Theta}, & \text{if } 2\Theta < \frac1p,\\
W_p^{2\Theta}(\Omega)\cap \ker \tr_{\partial \Omega}, & \text{if } 2\Theta > \frac1p.
\end{cases}
\end{align} 
He actually gives a proof for any normal boundary system (defined in the sense of \cite{Seeley}, \textsection 3). To determine $[Y_0,D(\A)]_\Theta$ we start with the case of uniform
 spherical cells $B$. Let $0<\Theta<1$. 
 We set
\begin{align*}
 X_0 &= L_p(\Omega) \times L_p(\Omega, L_p(B)),\\
 X_1 &= \left( W_p^2(\Omega) \cap \ker \tr_\Gamma\right) \times L_p(\Omega, W_p^2(B)\cap \ker \tr_S).
\end{align*}
Then $\{X_0,X_1\}$ is an interpolation couple. Due to Proposition I.2.3.3 in \cite{Amann} it suffices to interpolate both factors separately. Let $2\Theta > \frac1p$. We deduce from \eqref{15}
$$  \left[ L_p(\Omega), W_p^2(\Omega)\cap \ker \tr_\Gamma\right]_\Theta = \ker \tr_\Gamma \cap W_p^{2\Theta} (\Omega).$$
Further the results in \cite{Calderon} and \eqref{15} show that
\begin{align*}
\left[L_p(\Omega, L_p(B)), L_p(\Omega, W_p^2(B) \cap \ker \tr_S)\right]_\Theta
 &= L_p(\Omega, \left[ L_p(B), W_p^2(B)\cap \ker \tr_S\right]_\Theta ),\\
&= L_p(\Omega, W_p^{2\Theta}(B)) \cap \ker \tr_S.
\end{align*}
The map $(I, \Phi_*): (u,U) \mapsto (u, \Phi_* U)$ is an isomorphism. It maps
\begin{align*}
  X_0 = L_p(\Omega) \times L_p(\Omega, L_p(B)) \to L_p(\Omega)\times L_p(\Omega, L_p(\Omega_x))=: Y_0
\end{align*}
as well as 
 $$ X_1 \to (W_p^2(\Omega)\cap \ker \tr_\Gamma) \times L_p(\Omega, W_p^2(\Omega_x)) \cap \ker \tr =: Y_1.$$
So $\{Y_0,Y_1\}$ is an interpolation couple and Proposition I 2.3.2 from \cite{Amann} implies 
\begin{align*}
\left[ Y_0,Y_1\right]_\Theta &= (I, \Phi_*) \left[X_0,X_1\right]_\Theta= \left(W_p^{2\Theta}(\Omega)\cap \ker \tr_\Gamma\right) \times \left(L_p(\Omega, W_p^{2\Theta}(\Omega_x)) \cap \ker \tr\right)
\end{align*}
for $2\Theta > \frac1p$. Finally we define the isomorphism
\begin{align*}
 J: Y_0 &\to Y_0: (u,U) \mapsto (u, U + Ru).
\end{align*}
Here $R$ is the retraction of the lifted trace. Then $J$ maps $Y_1$ onto $D(\textbf{A})$. Clearly
$T: Y_0\to Y_0 : (u,U) \mapsto (u, U- Ru)$ 
is the inverse of $J$. So $J$ fulfills the conditions of Proposition I 2.3.2 from \cite{Amann} for the interpolation couples $\{Y_0,Y_1\}$ and $\{Y_0, D(\textbf{A})\}$. Hence it maps $ \left[Y_0,Y_1\right]_\Theta$ onto $\left[Y_0, D(\textbf{A})\right]_\Theta.$ So for $2 \Theta > \frac1p$, we get explicitely
\begin{align*}
 \left[Y_0, D(\textbf{A})\right]_\Theta &= J\left( \left[Y_0,Y_1\right]_\Theta\right)= \bigcup_{\begin{subarray}{l}u\in W_p^{2\Theta} (\Omega)\\ \quad \cap \ker \tr_\Gamma\end{subarray} } \{u\} \times \{U\in L_p(\Omega, W_p^{2\Theta}(\Omega_x)); \tr U = u\}.
\end{align*}
If $2 \Theta < \frac1p$ the boundary condition drops in both scales. Hence we conclude
$$ \left[Y_0, D(\A)\right]_\Theta = W_p^{2\Theta}(\Omega)\times L_p(\Omega, W_p^{2\Theta}(\Omega_x)).$$
A similar analysis can be done for real interpolation functors. In particular
\begin{align}
(Y_0, D(\textbf{A}))_{1-\frac1p,p}& = \bigcup_{\begin{subarray}{l}u\in W_p^{2-\frac2p} (\Omega)\\ \quad \cap \ker \tr_\Gamma\end{subarray} } \{u\} \times \{U\in L_p(\Omega, W_p^{2-\frac2p}(\Omega_x)); \tr U = u\}.
\end{align}
Let $0<\Theta < 1$ and write $X^\Theta = [Y_0, D(\A)]_\Theta.$
We consider functions 
\begin{equation}\label{rechts}
 \underline{f}= (f,g): [0,\infty) \times  X^\Theta \to Y_0.
\end{equation}
\begin{thm}\label{allgemein}
 Let $\underline{f} = (f,g)$ be as in \eqref{rechts}. Assume 
$$\underline{f}\in C^{1-}([0,\infty) \times X^\Theta, Y_0),$$ 
is locally Lipschitz continuous for some $0<\Theta<1$. Then for any $(u_0,U_0) \in X^\Theta$, there exists 
$T=T(u_0, U_0, \Theta)>0$, such that (\ref{P}) has a unique strong solution $w=(u,U)$ on $(0,T)$ which satisfies initial conditions
\begin{align*}
 u(t=0)&= u_0 \qquad\text{ and }\qquad U(t=0)=U_0.
\end{align*}
In particular
\begin{align*}
w\in C^1\left(([0,T), Y_0\right)\cap C\left([0,T), X^\Theta\right). 
\end{align*}
\end{thm}
\begin{proof}
With the above considerations, Theorem 12.1 and Remark 12.2. (b) from \cite{NLQ} can be applied to the abstract equation (\ref{P}). The regularity results are proved in 
\cite{Amann84}.
\end{proof}
In order to state our main result on the matched microstructure model let
$$X_1^\Theta := W_p^{2\Theta} (\Omega)\cap \ker \tr_{\partial \Omega}, \qquad \qquad \Theta > \frac{1}{2p},$$
denote the first component of the interpolation space $X^\Theta$. 
\begin{cor} 
Let $$f\in C^{1-}\left([0,\infty)\times X_1^{1/2+1/2p}, L_p(\Omega)\right)$$ be given. Then for $(u_0, U_0)\in X^{1/2+1/2p}$, there exists $T>0$, such that the matched microstructure 
problem (MM) has a unique strong solution on $(0,T)$.
\end{cor}
\begin{proof}
Let $\Theta=\frac12+\frac{1}{2p}$. Let $U\in L_p(\Omega, W_p^{2 \Theta}(\Omega_x))$, $V= \Phi_* U$. 
 Since $q(\cdot)$ is time independent it remains to show the Lipschitz continuity in $U$. There is $C>0$, such that
 \begin{align*}
& \|q(U)\|^p_{L_p(\Omega)}\\ &\leq \int_\Omega  \left| \int_S \sqrt{|g(x,s)|} g^{ij}(x,s) \partial_{y_i} \hat{V}(x,s) \nu_j \, ds \right|^p \, dx\\
&\leq c^p_p\max_{\tiny{\begin{array}{c}
                  (x,y)\in \overline{\Omega}\times \overline{B},\\ i,j
                 \end{array}}}
 \left\{\sqrt{|g|(x,y)}\left|g^{ij}(x,y)\right|\right\}^p \int_\Omega \sum_{k=1}^n \int_S |\partial_{y_k} \hat{V}(x,s)|^p \,ds \, dx \\
&\leq C \int_\Omega \|\hat{V}(x)\|^p_{W_p^{2\Theta}(B)}\, dx = C\|U \|^p_{L_p(\Omega, W_p^{2\Theta} (\Omega_x))}.
 \end{align*}
The constant $c_p$ is the embedding constant of $L_p(B)$ into $L_1(B)$. Together with the linearity of $q$ this shows that $q$ is locally Lipschitz continuous in $$W_p^{1+\frac12}(\Omega)\times L_p(\Omega, W_p^{1+\frac12}(\Omega_x)) \supset X^\Theta.$$ 
Setting finally $\underline{f}(u,U) := (f(u)+q(U),0)$, the assumption follows from Theorem \ref{allgemein}.
\end{proof}
\subsection{Exponential Decay under Dirichlet Boundary Conditions, Neumann BC}\label{qualitativ}
Now we assume that there are no external sources in the system, i.e. $f=0$. We will show that the corresponding solutions decay exponentially fast to zero. First we investigate the spectrum of $\A$.
Let $p=2$. The space $L_2(\Omega) \times L_2(\Omega, L_2(B))$ is a Hilbert space. This implies that $Y_0$ is a Hilbert space with the inner product
$$((u,U), (w,W))_{Y_0} = (u,w)_{L_2(\Omega)} + (\Phi_* U, \Phi_* W)_{L_2(\Omega\times B)}$$
for $(u,U), (w,W) \in Y_0$.
We introduce an extended operator $\A_q$ by
\begin{align*}
 dom(\A_q) &= dom (\A),\\
 \A_q (u,U) & = \A(u,U) - (q(U),0), && \text{for all } (u,U)\in dom(\A_q).
\end{align*}
We investigate the spectrum of $\A_q$. It is convenient to introduce a weighted space $Y_g$. Set
\begin{align*}
 Y_g&= L_2(\Omega)\times L_2(\Omega, L_2(\Omega_x, \sqrt{|g|})),\\
\|(u,U)\|^2_{Y_g} &= \|u\|^2_{L_2(\Omega)} + \|\Phi_* U \|^2_{L_2( \Omega, L_2(B,\sqrt{|g|}))}.
\end{align*}
This is a well defined Hilbert space with respect to the inner product 
\begin{align*}
 ((u,U), (w,W))_{Y_g} &= \int_\Omega uw + \int_{\Omega\times B} \sqrt{|g|}\Phi^* U \Phi^* V,&&\text{for }(u,U), (w,W) \in Y_g.
 \end{align*} We can show that with this modified inner product $Y_g$ is also a Hilbert space. 
\begin{lem}
The operator $\A_q$ is self adjoint in $Y_g$.
\end{lem}
\begin{proof}
Take $(u,U), (w,W) \in dom(\A_q)$. Let $V= \Phi^* U$, $Z=\Phi^* W$. 
 Then
$$((u,U), (w,W))_{Y_g} = (u,w)_{L_2(\Omega)} + (V,Z)_{L_2(\Omega, L_2(B, \sqrt{|g|}))}$$
and
\begin{align}
 &(\A_q(u,U),(w,W))_{Y_q} =  \int_\Omega \left(-\Delta_x u(x) w(x) - q(U)(x)  w(x)\right) \, dx\nonumber \\
&\qquad  - \int_\Omega \int_B \frac{1}{\sqrt{|g|}} \sum_{i,j} 
\partial_{y_i} \left(g^{ij} \sqrt{|g|} \partial_{y_j} V(x,y)\right) Z(x,y) \sqrt{|g|} \, dy \, dx.\label{Beweis}
\end{align}
Manipulation of the last integral in (\ref{Beweis}) by partial integration together with the boundary conditions on $\Omega$ and $B$ shows
\begin{align*}
& -\int_\Omega \int_B \frac{1}{\sqrt{|g|}} \sum_{i,j} 
\partial_{y_i} \left(g^{ij} \sqrt{|g|} \partial_{y_j} V(x,y)\right) Z(x,y) \sqrt{|g|} \, dy \, dx\\
&= -\int_\Omega \int_B   V(x,y) \sum_{i,j}\partial_{y_j}\left( g^{ij} \sqrt{|g|}\partial_{y_i} Z(x,y)\right) \, dy \, dx + \int_\Omega q(U) w - \int_\Omega u\ q(W).
\end{align*}
This implies that
$\A_q$ is symmetric. From the theory of elliptic operators and the representation of $\A$ we can conclude that $\A$ is invertible, in particular we have that $im \A= Y$. Our next step is to show that this holds for $\A_q$ as well. Let $(v,V)\in Y$. We show that there exist $(z,Z)\in dom(A_q)$ with $\A_q(z,Z)=(v,V)$. First we know that there are $(u,U) \in dom(\A)$ such that $\A (u,U)= (v,V)$. Then $$\A_q(u,U) = (v-q(U),V).$$ Clearly $(q(U),0)\in Y$. Again there are functions $(w,W)\in dom( \A)$ with $\A(w,W) = (q(U),0)$. This implies that $W(x) =const.$ for a.e. $x\in \Omega$. Thus $q(W)=0$. Then from the linearity of $\A_q$ it follows that
 \begin{align*}
  \A_q\{ (w,W)+(u,U)\} =(q(U),0) + (v-q(U), V) = (v,V).
 \end{align*}
Thus $\A_q$ is symmetric and $im(\A_q)= Y$. It follows from the fact that
$$ \ker (\A_q^*) = im (\A_q)^\perp = \{0\},$$
that the dual operator is injektiv. Thus $\A_q \subset \A_q^*$ implies the assertion.
\end{proof}
\begin{lem}
 It exists a constant $\sigma>0$ such that
$$ (-\A_q(u,U), (u,U))_{Y_g} \leq-\sigma ((u,U),(u,U))_{Y_g},\, \text{ for all }(u,U)\in dom(\A_q).$$
\end{lem}
\begin{proof}
Let $(u,U)\in dom(\A_q)$. Set $V= \Phi^* U.$ We make use of equivalent norms in $W_2^1$, Sobolev embedding results and that the metric $g^{ij}$ is bounded from below. Let $C>0$ denote an appropriate constant. It holds
\begin{align*}
& (-\A_q(u,U), (u,U))_{Y_g}\\&= -\int_\Omega |\nabla_x u|^2 + \int_\Omega q(U) u -  \int_\Omega \int_B \sum_{i,j} g^{ij} \sqrt{|g|} (\partial_{y_i}V) (\partial_{y_j} V)\\
&\qquad + \int_\Omega \underbrace{\int_S \sum_{i,j} g^{ij} \sqrt{|g|} \partial_{y_i}V\cdot \nu_j}_{= -q(U)} \overbrace{V(x)}^{= u(x)}\\
&\quad\leq -C\left(\|u\|^2_{L_2(\Omega)} + \|U\|^2_{L_2(\Omega,L_2(\Omega_x))}\right)= -\sigma\left((u,U),(u,U)\right).
\end{align*}
Hence we have obtained a bound for the numerical range of $-\A_q$ in the weighted space.
\end{proof}
The spectrum of a self adjoint operator is contained in the closure of its numerical range (see \cite{Kato}, Section V, \textsection 3). Hence the spectrum of $-\A_q$ lies totally on the right hand side of $-\sigma$. Since the weighted norm and the usual norm on $Y$ are equivalent, we also get a spectral
bound for $-\A_q$ in the unweighted space. 
So the right half space is containt in the resolvent of $-\A_q$. We set
\begin{align*}
 Q: D(\A^{\frac12}) \to Y_0 : (u,U) \mapsto (-q(U),0).
\end{align*}
Then $Q\in \mathcal{L}(Y_{\frac12},Y_0)$. Obviously
$$-\A_q = -\A + Q$$
and the conditions of Proposition 2.4.1 in \cite{Luna} are satisfied. So $\A_q$ is sectorial. The matched microstructure problem is equivalent to
\begin{align}\label{MMPq}
\begin{cases}
 \partial_t(u,U) + \A_q (u,U) =0, \qquad t\in (0,T),\\
(u,U)(0)= (u_0,U_0).
\end{cases}
\end{align}
\begin{prop}\label{Prop_decay}
 Let $(u,U)$ be a solution of (\ref{MMPq}).\\ 
 Then $(u,U) \searrow (0,0)$ exponentially fast.
\end{prop}
\begin{proof}
This follows from the the fact that for analytic semigroups the growth bound and the spectral bound coincide as it is e.g. shown in \cite{EngNa} Corollary 3.12. 
\end{proof}
Proposition \ref{Prop_decay} allows also to apply the principle of linearized stability to the semilinear version of (MM), provided $f$ is of class $C^1$, cf. \cite{Luna}.
We can also treat a modified model with no-flux or Neumann boundary conditions on $\Omega$. Thus we want 
$\partial_\nu u = 0  \text{ on }\Gamma.$
We set
\begin{align*}
 dom(A_1^N) &= \left\{u\in W_p^2(\Omega); \partial_\nu u =0\text{ on }\Gamma\right\},\\
 A_1^N u &= -\Delta_x u, &&\text{for all } u\in dom(A_1^N).
\end{align*}
The boundary conditions in the cells are not changed. Hornung and J\"ager  also derived this model in \cite{HorJa}.
Then $-A_1^N$ is the generator of a strongly continuous, analytic semigroup in $L_p(\Omega)$ for $1<p<\infty$.
We set
\begin{align*}
 dom(\textbf{A}^N) &= \bigcup_{u\in dom(A_1^N)} \{u\} \times D_0(u),\\
 \textbf{A}^N (u,U) &= (A_1^N u, [\Phi_* \mathcal{A}_x \hat{V}(x)]), && \text{for } (u,U) \in dom (\textbf{A}^N), V= \Phi^* U. 
\end{align*}
The modified model can be formulated as evolution equation
\begin{align}\label{MMPn}
 \begin{cases}
 \partial_t(u,U) + \A^N (u,U) =(q(U),0),\qquad t\in (0,T),\\
(u,U)(0)= (u_0,U_0).
\end{cases}
\end{align}
The changes in the operator occur only on the macroscales. Thus they can be treated with well known results for elliptic operators on bounded domains. So the same considerations as for $\textbf{A}$ can be done for $\A^N$. Existence and uniqueness can be proved similarly as in Chapter \ref{3.3}. Nevertheless the qualitative behaviour is different.
\begin{prop}
 Let $(u_0,U_0)\in W_p^1(\Omega)\times L_p(\Omega, W_p^1(\Omega_x))$ and let $(u,U)$ be the solution to the matched microstructure problem with Neumann boundary conditions (\ref{MMPn}) on some time interval $[0,T]$. Then the material value
$$ S(u,U) := \int_\Omega u + \int_\Omega \int_B \sqrt{|g|} \Phi^* U, \qquad t\in (0,T], $$
is preserved.
\end{prop}
\begin{proof} This follows from a straight forward calculation
\end{proof}
\section{Nonlinear Boundary Conditions}\label{Gravity}
The following version of the MMP brings gravity into the scheme. Let us restrict ourselves to uniform cells. Assume we are given a function $f : (0,2\pi) \to (0,\infty)$ periodic and differentiable. We consider the fixed domain 
$$\Omega_f= \left\{(x,y) \in \mathbb{S }^1\times \R; 0<y<f(x)\right\}.$$ 
It is shown in Figure \ref{Omega_f}. 
\begin{figure}[htbp]
 \begin{center}
 \includegraphics[width=0.5\textwidth]{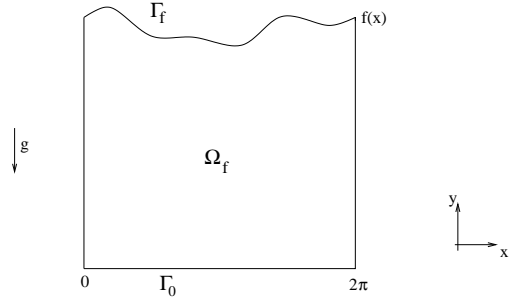}
 \caption{The periodic domain $\Omega_f$}\label{Omega_f}
\end{center}
\end{figure}
The gravitational force points into the $-y$ direction. The almost cylindrical domain $\Omega_f$ can be treated with the same methods as before. For a work on the torus see \cite{EsPro}. Let $h: \Omega_f\times (0,T) \to \R$ describe the sources and sinks in the macro system. A solution of the matched microstructure model with gravity is a pair of functions $(u,U)$ that satisfies
\begin{align*}
 (\text{P})\quad \left\{ \begin{array}{r cl l}
\partial_t u -\Delta_x u &=&  h + q(U), & \text{on } \Omega_f, t\in (0,T),\\
\partial_2 u &=&  -u^2, & \text{on } \Gamma_0, t\in(0,T),\\
u&=& \rho_0,& \text{on } \Gamma_f, t\in (0,T),\\
\partial_t U - \Delta_y U &=&  0,& \text{in } \Omega_f\times B, t\in (0,T),\\
U&=&u, & \text{on }\Omega_f\times S, t\in (0,T),\\
(u,U)(0) &=& (u_0,U_0), &\text{on }\Omega_f \times (\Omega_f\times B).
             \end{array}\right.
\end{align*}
Let $v= u - \rho_0 \cdot \mathds{1}_{\Omega_f}$, $ V= U - \rho_0 \mathds{1}_{\Omega_f\times B}$. By definition (\ref{qU}) it holds that $q(U) =q(V+\rho_0) = q(V)$. So $(v,V)$ solves
\begin{align*}
 (\text{P'})\quad \left\{ \begin{array}{r cl l}
\partial_t v -\Delta_x v &=&  h + q(V), & \text{on } \Omega_f, t\in (0,T),\\
\partial_2 v &=&  -(v+\rho_0)^2, & \text{on } \Gamma_0, t\in(0,T),\\
v&=& 0,& \text{on } \Gamma_f, t\in (0,T),\\
\partial_t V - \Delta_y V &=&  0,& \text{in } \Omega_f\times B, t\in (0,T),\\
V&=&v, & \text{on }\Omega_f\times S, t\in (0,T),\\
(v,V)(0) &=& (u_0-\rho_0,U_0-\rho_0), &\text{on }\Omega_f \times (\Omega_f\times B).
             \end{array}\right.
\end{align*}
To treat the nonlinear boundary condition in the macroscopic scale we us a weak formulation. As before we define operators $A_1$ and $\A_q$ with linear zero boundary conditions ($\partial_2 v = 0$ on $\Gamma_0$, $v=0$ on $\Gamma_f$). The mixed Dirichlet-Neumann conditions on $\Omega_f$ do not effect the properties of the operators. Especially $\A_q$ is selfadjoint and $-\A_q$ is the generator of an analytic semigroup in the Hilbert space 
$$Y_0= L_2(\Omega_f) \times L_2(\Omega_f\times B).$$
The operator $A_1$ is a well known form of the Laplace operator and so it is invertible. 
 The presented method is due to Amann \cite{Amann88} and Escher \cite{Escher89}. The main idea is to move the nonlinearity from the boundary to the right hand side $h$. Therefore we need to construct an appropriate inverse operator from the trace space on $\Gamma_0$ to the domain $\Omega_f$. The resulting semilinear evolution equation can be treated as before. 
For $u\in L_2(\Omega_f)$ we set
$$D_0^s(u) = \left\{ U\in L_2(\Omega_f, W_2^{2s}(B)), \tr_S U = u\right\}, \qquad \frac14<s<\infty.$$
In the rest of the chapter we drop the index $f$ of $\Omega_f$. Let $\tr_0, \tr_f$ denote the trace operators onto $\Gamma_0$ and $\Gamma_f$. With the boundary operator on $\Omega$ we mean an operator $\mathcal{B}$, with
$\mathcal{B} u = \tr_0 \partial_\nu u + \tr_f u.$ 
We define
\begin{align}
 Y_s = \begin{cases}
        \left\{ (u,U), u\in W_2^{2s}(\Omega), U\in D_0^s(u), \mathcal{B}u=0 \right\}, & \text{for } \frac32 < 2s\leq \infty,\\
 	\left\{ (u,U), u\in W_2^{2s}(\Omega), U\in D_0^s(u), \tr_f u=0\right\}, & \text{for } \frac12 < 2s \leq \frac32,\\
	W_2^{2s}(\Omega)\times L_2(\Omega, W_2^{2s}(B)), & \text{for } 0\leq2s\leq \frac12.
       \end{cases}\label{spaces1}
\end{align}
To construct a suitable retract, we first restrict ourselves to the macro scale.
Let $\mathcal{A}_1= - \Delta_x $. Considered as an unbounded operator in $L_2(\Omega)$ it is closable. Together with $\mathcal{B}$ it fits into the scheme of \cite{Escher89}, Chapter 3. We will use the same notation. Let $\overline{\mathcal{A}}_1$ be the closure of $\mathcal{A}_1$. Then $W_2^2(\Omega)\stackrel{\mathrm{d}}\hookrightarrow D(\overline{\mathcal{A}}_1)$. In addition we set
$\mathcal{C} u = \tr_0 u+ \tr_f \partial_\nu u,$
and
\begin{align*}
 \partial W_2^{2s} = W_2^{2s-\frac32}(\Gamma_0) \times W_2^{2s-\frac12}(\Gamma_f),&&\partial_1 W_1^{2s} = W_2^{2s-\frac12}(\Gamma_0) \times W_2^{2s-\frac32}(\Gamma_f),
\end{align*}
for $0\leq s\leq 1$. Combined together, the map $(\mathcal{B},\mathcal{C})\in \L(W_2^2(\Omega), \partial W_2^2\times \partial_1 W^2_2)$ is a retraction. So we can apply Theorem 4.1. from \cite{Amann88}:
\begin{prop}
 There exists a unique extension $\left(\overline{\mathcal{B}}, \overline{\mathcal{C}}\right)\in \L(D\left(\overline{\mathcal{A}}_1\right), \partial W_2^0\times \partial_1 W^0_2)$ of $(\mathcal{B},\mathcal{C})$ such that for $u\in D(\overline{\mathcal{A}}_1)$, $v\in W_2^2(\Omega)$ 
 the generalized Green's formula
$$\langle v, \overline{\mathcal{A}}_1 u \rangle_{Y_0} + \langle \mathcal{C}v, \overline{\mathcal{B}}u\rangle_{\partial W_2^0} = \langle \mathcal{A}_1 v, u\rangle_{Y_0} + \langle \mathcal{B}v,  \overline{\mathcal{C}}u\rangle_{\partial_1 W_2^0}$$ is valid.
\end{prop}
Then from interpolation theory (Proposition I 2.3.2 in \cite{Amann}) and well known a priori estimates for $\mathcal{A}_1$, it follows that
 $$\left(\overline{\mathcal{A}}_1, \overline{\mathcal{B}}\right) \in Isom((D(\overline{\mathcal{A}}_1),W_2^2(\Omega))_\theta, L_2(\Omega)\times \partial W_2^{2\theta}), \qquad \theta\in [0,1].$$
Therefore we can define the right inverse 
$$R_\theta = \left(\overline{\mathcal{A}}_1, \overline{\mathcal{B}}\right)^{-1}| \{0\} \times \partial W_2^{2\theta}, \qquad \theta \in [0,1].$$
Then $R_\theta \in \L(\partial W_2^{2\theta}, W_2^{2\theta}(\Omega))$. We now add the microscopic scale. 
With $\Y^s$ we mean
\begin{align*}
 \Y^s = \begin{cases}
         \{(u,U)\in W_2^{2s}(\Omega)\times L_2(\Omega, W_2^{2s}(B)); U\in D_0^{2s}(u)\}, & \frac12 < 2s\leq 2,\\
	W_2^{2s}(\Omega)\times L_2(\Omega, W_2^{2s}(B)), & 0\leq  2s \leq \frac12.
        \end{cases}
\end{align*}
So if $s< \frac34$ the two sets $Y_s$ and $ \Y^s$ coincide.
Define $\textbf{R}_\theta \in \L(\partial W_2^{2\theta},\Y^{2\theta})$ by 	 
$$\textbf{R}_\theta u= (R_\theta u, R_\theta u \cdot \mathds{1}_B), \qquad u\in \partial W_2^{2\theta}.$$
We set
$$\partial_0 W_2^{2\theta} = \{u\in \partial W_2^{2\theta}; \tr_f u =0\}.$$
Obviously it is a closed linear subspace of $\partial W_2^{2\theta}$. It can be identified with the space $W_2^{2\theta-\frac32}(\Gamma_0)$. So it holds
$$\textbf{R}_\theta( \partial_0 W_2^{2\theta}) \subset Y_{2\theta},\qquad \text{if }2\theta \leq \frac32.$$
For the formulation of the abstract evolution problem we use the scale of interpolation and extrapolation spaces $\{(Y_\alpha, (\A_q)_\alpha ), \alpha \in \R\}$ as it was defined by Amann \cite{Amann88}. For $0\leq\alpha \leq 1$ this corresponds to the interpolation spaces in Section \ref{3.3} and Definition (\ref{spaces1}). 
 We set as in \cite{Escher89}
\begin{align*}
 \mathbb{A} = (\A_q)_{-\frac12}, \qquad H = Y_{-\frac12}, \qquad    D=Y_{\frac12} = D(\mathbb{A}).
\end{align*}
Then the duality theory tells us that $D= H'$ and the duality pairings satisfy
$$\langle \vec{u},\vec{v}\rangle_H = \langle \vec{u},\vec{v} \rangle_{Y_0} = (\vec{u}|\vec{v})_{Y_0}, \quad \text{ for } \vec{u}\in D, \vec{v}\in Y_0.$$
Let  $a: D\times D \to \R$ be the coercive bilinear form
$$a(\vec{u},\vec{v}) = \int_\Omega \nabla_x u \cdot \nabla_x v \, dx + \int_{\Omega\times B} \nabla_z U \cdot \nabla_z V \, d(x,z), \quad \vec{u}, \vec{v} \in D.$$
Usually it is clear from the context whether we refer to the function $u$ living on $\Omega$ or the pair $\vec{u}=(u,U)$. For $\vec{u},\vec{v}\in D$ we get
\begin{align*}
  \langle \vec{v}, \mathbb{A} \vec{u}\rangle_H 
&= \int_\Omega \nabla_x v \cdot \nabla_x u - \int_\Omega  v q(U) + \int_{\Omega\times B } \nabla_z V \cdot \nabla_z U   - \int_{\Omega\times S} V \nabla_z U \cdot \nu\\ 
& \qquad- \int_{\Gamma_0} v \partial_\nu u - \int_{\Gamma_f} v \nabla_x u \cdot \nu = a(\vec{u},\vec{v}). 
\end{align*}
The possible approximation of $u$ by functions in $Y_1$ and the continuity of the left and right hand side justify this formal calculation. 
To treat the nonlinear boundary condition we define the map $G: D \to L_2(\Gamma_0)$,
$$G(u) = -\tr_0(u+\rho_0)^2.$$
We have to show that $(h,G)$ satisfies the assumption (3.6) of \cite{Escher89}. For $h$ we just assume that $h\in Y_0$. For $G$ the properties are summarized in the following lemma.
\begin{lem} \label{Lemma22}
 $$G\in C^1(D, W_2^{2\beta+\frac12}(\Gamma_0))$$ 
for any fixed $\beta \in (-\frac12,-\frac14)$,
and the Lipschitz continuity is uniform on bounded sets. 
\end{lem}
\begin{proof}
  Fix $\beta\in \left(-\frac12, -\frac14\right)$.
Let $\vec{u}\in D$. Then also $u+ \rho_0 \in W_2^1(\Omega)$. From \cite{Amann_Multi}, Theorem 4.1 and the fact that the Besov space $B_{22}^s(\Omega) = W_2^s(\Omega)$, we know that the multiplication
$ W_2^1( \Omega) \cdot W_2^1(\Omega) \to W_2^{1-\epsilon}(\Omega)$
is continuous for $0<\epsilon < 1$. We conclude that for fixed $\epsilon < \frac12$ we have
$$(u+\rho_0)^2 \in W_2^{1-\epsilon}(\Omega) \qquad \text{and} \qquad \tr_0(u+\rho_0)^2 \in W_2^{\frac12-\epsilon}(\Gamma_0) .$$ 
Then by Sobolev embedding it holds $ W_2^{\frac12-\epsilon}(\Gamma_0) \stackrel{\mathrm{d}}\hookrightarrow L_2(\Gamma_0) \stackrel{\mathrm{d}}\hookrightarrow W_2^{2\beta+ \frac12}(\Gamma_0).$
The second inclusion follows from the definition of $W_2^{-s}$ as a dual spaces for $s>0$.
So finally
$$-\tr_0 (u+\rho_0)^2 \in W_2^{2\beta +\frac12}(\Gamma_0).$$
The Fr\'echet derivative of $G$ is the linear operator
$\partial G(u)v = -2\tr_0 (u+\rho_0) v.$
Thus $G\in C^1(D, W_2^{2\beta+\frac12}(\Gamma_0))$. It remains to show that the map is uniformly Lipschitz continuous on bounded sets. Let $W\subset D$ be bounded. Take $\vec{u},\vec{v}\in W$. Then
\begin{align*}
\tr_0 (u+\rho_0)^2 -\tr_0 (v+\rho_0)^2 = \tr_0 u^2 -\tr_0 v^2 + 2\rho_0 (\tr_0 u-\tr_0 v).
\end{align*}
Clearly the last term is uniformly Lipschitz on $W$.
Further $W_2^1(\Omega) \hookrightarrow C(\overline{\Omega})$. So a bounded set in $D$ is bounded in $C(\overline{\Omega})$. Thus there exists a constant $c_1>0$, such that $\|u\|_\infty \leq c_1$ for all $\vec{u}\in W$. It follows from this and Sobolev embeddings that
\begin{align*}
 \| \tr_0 u^2 -\tr_0 v^2\|_{W_2^{2\beta+\frac12} (\Gamma_0)}& \leq C \|\tr_0(u^2-v^2)\|_{L_2(\Gamma_0)}\\
& \leq L \|u-v\|_{W_2^1(\Omega)}. 
\end{align*}
The last constant $L$ is independent of $\vec{u},\vec{v}\in W$. This completes the proof.
\end{proof}
Now we define the right hand side to write (P') as an abstract evolution equation. Let $\textbf{R} :=\textbf{R}_{\frac12}$.
Then it holds for $u \in W_2^{\frac12}(\Gamma_0)$ and $ \vec{v}\in D$
\begin{align}
 \langle \vec{v}, \mathbb{A} \textbf{R} u\rangle_H = \langle v, u\rangle_{W_2^{\frac12}(\Gamma_0)}=: \langle v,u\rangle_{\Gamma_0},\label{Rand}
\end{align} 
in the sense of trace. We set
\begin{align*}
 F(\vec{u}) = (h,0)+ \mathbb{A} \textbf{R} G(\vec{u}), \qquad \vec{u}\in D.
\end{align*}
Note that the second component of $F(\vec{u})$ vanishes since $R_{\frac12}u\cdot \mathds{1}_B$ is constant on each cell. By assumption $h\in Y_0 \hookrightarrow Y_\beta$. It was shown in \cite{Escher89}, p.301, that under this circumstances $F$ is well defined and the previous lemma ensures that 
$$F\in C^1(D, Y_\beta) \text{ is uniformly Lipschitz continuous on bounded sets.}$$
\begin{prop}\label{Prop23}
 For each $\vec{u}_0\in D$ there is a unique maximal solution $\vec{u}(\cdot,\vec{u}_0)\in C([0,T_1), D)$ of the semilinear Cauchy problem 
\begin{align}
 \dot{\vec{u}} + \mathbb{A}\vec{u} = F(\vec{u}), \qquad \vec{u}(0)=\vec{u}_0\label{abstract}
\end{align}
with $0<T_1\leq \infty$. In addition for any $\epsilon \in (0,\frac14)$ holds
$$\vec{u}\in C((0,T_1), Y_{\epsilon + \frac12} )\cap C^1((0,T_1), Y_{\epsilon -\frac12}).$$
\end{prop}
\begin{proof}
 Take $\beta = -\frac12 +\epsilon$. Then the assertion follow from \cite{Amann88}, Sect. 12 and the previous lemma. 
\end{proof}
\renewcommand{\phi}{\varphi}
By a \textit{weak solution} of (P') we mean a function $\vec{u} \in C^1([0,T_1), D)$ such that the initialcondition $\vec{u}(0)= \vec{u}_0-\rho$ is satisfied, and  
$$ -\int_0^T \langle \dot{\phi},\vec{u}\rangle_H + a(\phi,\vec{u}) \, dt = \int_0^T  \left(\langle \phi, (h,0)\rangle_H +\int_{\Gamma_0} \phi G(\vec{u})\right)\, dt + \langle \phi(0), \vec{u}_0-\rho_0\rangle_H $$
for all $0<T<T_1$, $\phi \in C([0,T], D)\cap C^1([0,T],H)$ with $\phi(T)=0$. So the above considerations show that
\begin{cor}\ \\
 For each $\vec{u}_0\in Y_\frac12$ there exists a unique maximal weak solution of (P').
\end{cor}
\begin{proof}
 This follows from the representation of $\mathbb{A}$ and (\ref{Rand}).
\end{proof}
\textbf{Remarks:}\\ 
(i) The construction in \cite{Escher89} allows to consider a more general $h\in C^1(D,Y_\beta)$. This means a full semilinear version of (P') can be treated.\\ 
(ii) Abstract results on evolution equations in interpolation-extrapolation scales ensure that the solutions satisfies
$$(u,U) \in C((0,T_1),Y_1)\cap C^1(0,T_1),Y_0).$$
 So the system (P') is satisfied pointwise in time.

\ \\
\
\hspace{1cm}
\tt
Institute for Applied Mathematics,
Leibniz University Hanover,
Welfengarten 1,
D 30167 Hanover,
Germany\\
 \textbf{email:} escher@ifam.uni-hannover.de\\
Institute for Applied Mathematics,
Leibniz University Hanover,
Welfengarten 1,
D 30167 Hanover,
Germany\\
\textbf{email:} treutler@ifam.uni-hannover.de

\end{document}